\documentclass[12pt]{amsart}
\usepackage{amsfonts}

\newtheorem{theorem}{Theorem}[section]
\newtheorem{lemma}[theorem]{Lemma}
\newtheorem{proposition}[theorem]{Proposition}
\newtheorem{corollary}[theorem]{Corollary}
\newtheorem{hypothesis}[theorem]{Hypothesis}

\theoremstyle{remark}

\newtheorem*{acknowledgments}{Acknowledgments}

\newcommand{\et}{\quad\mbox{and}\quad}
\newcommand{\Adj}{\text{Adj}}

\newcommand{\bN}{\mathbb{N}}

\newcommand{\bQ}{\mathbb{Q}}
\newcommand{\bR}{\mathbb{R}}
\newcommand{\bZ}{\mathbb{Z}}

\newcommand{\cL}{{\mathcal{L}}}

\newcommand{\gA}{{\mathfrak{A}}}
\newcommand{\gB}{{\mathfrak{B}}}
\newcommand{\gM}{{\mathfrak{M}}}
\newcommand{\gS}{{\mathfrak{S}}}
\newcommand{\gX}{{\mathfrak{X}}}

\newcommand{\SL}{\mathrm{SL}}

\newcommand{\ud}{\mathbf{d}}
\newcommand{\ui}{\mathbf{i}}
\newcommand{\uj}{\mathbf{j}}
\newcommand{\un}{\mathbf{n}}
\newcommand{\ux}{\mathbf{x}}
\newcommand{\uX}{\mathbf{X}}

\newcommand{\tM}{{\,{^t}\hskip-2pt M}}

\newcommand{\QgX}{\bQ\big[\gX^{(0)},\,\gX^{(-1)}\big]}
\newcommand{\QX}{\bQ[\uX,\uX^*]}

\begin{document}

\baselineskip=16.8pt

\title[On the ring of approximation triples]
{On the ring of approximation triples\\ attached to a class of
extremal real numbers}

\author{Damien ROY}
\address{
   D\'epartement de Math\'ematiques\\
   Universit\'e d'Ottawa\\
   585 King Edward\\
   Ottawa, Ontario K1N 6N5, Canada}
\email{droy@uottawa.ca}
\author{\'Eric Villani}
\address{
   D\'epartement de Math\'ematiques\\
   Universit\'e d'Ottawa\\
   585 King Edward\\
   Ottawa, Ontario K1N 6N5, Canada}
\email{evillani@uottawa.ca}
\thanks{Work partially supported by NSERC and CICMA}

\begin{abstract}
We attach a ring of sequences to each number from a certain class of
extremal real numbers, and we study the properties of this ring both
from an analytic point of view by exhibiting elements with specific
behaviors, and also from an algebraic point of view by identifying
it with the quotient of a polynomial ring over $\bQ$. The link
between these points of view relies on combinatorial results of
independent interest. We apply this theory to estimate the dimension
of a certain space of sequences satisfying prescribed growth
constrains.
\end{abstract}

\maketitle


\section{Introduction}
\label{sec:intro}


Let $\gamma= (1+\sqrt{5})/2$ denote the golden ratio.  In \cite{DS},
H.~Davenport and W.~M.~Schmidt proved that, for each real number
$\xi$ which is neither rational nor quadratic irrational, there
exists a constant $c>0$ with the property that, for arbitrarily
large real numbers $X$, the system of inequalities
\begin{equation}
 \label{ineq:DS}
 |x_0|\le X,
 \quad
 |x_0\xi-x_1| \le c X^{-1/\gamma},
 \quad
 |x_0\xi^2-x_2| \le c X^{-1/\gamma}
\end{equation}
has no non-zero solution $\ux = (x_0,x_1,x_2) \in \bZ^3$. Because of
this, we say that a real number $\xi$ is \emph{extremal} if it is
neither rational nor quadratic irrational and if there exists a
constant $c'>0$ such that the system \eqref{ineq:DS} with $c$
replaced by $c'$ has a non-zero solution $(x_0,x_1,x_2) \in \bZ^3$
for each $X\ge 1$.  The existence of such numbers is established in
\cite{Ra,Rb}, showing in particular that the exponent $1/\gamma$ in
the result of Davenport and Schmidt is best possible.  Among these
numbers are all real numbers whose continued fraction expansion is
the infinite Fibonacci word constructed on an alphabet consisting of
two different positive integers \cite{Ra} (or a generalized such
word constructed on two non-commuting words in positive integers
\cite{Re}).  This connection with symbolic dynamics is extended by
M.~Laurent and Y.~Bugeaud in \cite{BL}, and stressed even further in
recent work of S.~Fischler \cite{Fa,Fb}.

Schmidt's subspace theorem implies that any extremal real number is
transcendental (see \cite[Chap.~VI, Theorem 1B]{Sc}).  Using a
quantitative version of the subspace theorem, B.~Adamczewski and
Y.~Bugeaud even produced a measure of transcendence for theses
numbers, showing that, in terms of Mahler's classification, they are
either $S$ or $T$ numbers \cite[Theorem 4.6]{AB}.  The purpose of
the present paper is to provide tools which may eventually lead to
sharper measures of approximation to extremal real numbers either by
all algebraic numbers or by more restricted types of algebraic
numbers (like in \cite{Rc}).

As shown in \cite{Rb}, any extremal real number comes with rigid
sequences of integer triples $(x_0,x_1,x_2)$ satisfying a stronger
approximation property than that required by \eqref{ineq:DS}. In the
next section, we show that, for each extremal real number in some
large family, this naturally gives rise to a finitely generated ring
of sequences over $\bQ$.  We study this ring in Sections
\ref{sec:ring} and \ref{sec:bideg}, both from an analytic point of
view by exhibiting elements with specific behaviors, and also from
an algebraic point of view by showing that it is isomorphic to the
quotient of a polynomial ring in six variables over $\bQ$ by an
ideal $I$ with three explicitly given generators. The link between
these points of view relies on two combinatorial results of
independent interest that are stated at the beginning of Sections
\ref{sec:ring} and \ref{sec:bideg}, and proved in Section
\ref{sec:comb}.  In Section \ref{sec:applic}, we apply this theory
to estimate the dimension of a certain space of sequences with
restricted growth. Following a suggestion of Daniel Daigle, we
conclude in Section \ref{sec:complement} with a complementary result
showing that the ideal $I$ mentioned above is a prime ideal of rank
$3$ and thus, that the ring of sequences attached to the extremal
real numbers under study is an integral domain of transcendence
degree $3$.

%
%

\section{The ring of approximation triples}
 \label{sec:ring}


\subsection{A combinatorial result}
 \label{subsec:2.1}
We denote by $\bN$ the set of non-negative integers, and by $\bN^*=
\bN\setminus\{0\}$ the set of positive integers. We also denote by
$f\colon\bZ\to \bZ$ the function satisfying
\begin{equation}
 \label{def:f}
 f(0)=f(1)=1
 \et
 f(i+2)=f(i+1)+f(i)
 \quad \text{for each $i\in\bZ$,}
\end{equation}
so that $(f(i))_{i\ge 0}$ is simply the Fibonacci sequence.  In
order not to interrupt the flow of the discussion later, we start by
stating the following crucial combinatorial result whose proof is
postponed to Section \ref{sec:comb} and whose relevance will be made
clear shortly.

\begin{theorem}
 \label{thm:comb_deg}
Let $d\in\bN$.  For each $s\in\bN$, denote by $\chi_d(s)$ the number
of points $(m,n)\in\bZ^2$ for which the conditions
\begin{equation}
 \label{conditions:thm:comb_deg}
 (m,n) = \sum_{k=1}^s \big( f(-i_k),\, f(-i_k-1) \big)
 \et
 \sum_{k=1}^s f(i_k) \le d
\end{equation}
admit a solution integers $0\le i_1\le\cdots\le i_s$, and $s$ is
maximal with this property.  Then, 
\begin{equation}
 \label{formula:chi_d}
 \chi_d(s)
 = \begin{cases}
     2s+1 &\text{if\, $0\le s < d$,}\\
     d+1 &\text{if\, $s = d$,}\\
     0 &\text{if\, $s>d$.}
   \end{cases}
\end{equation}
\end{theorem}

Here, and throughout the rest of this paper, we agree that an empty
sum is zero, so that for $(m,n)=(0,0)$, the conditions
\eqref{conditions:thm:comb_deg} are satisfied with $s=0$.

For our purposes, we need to recast this result in the following
context.  Consider the subring $\bZ[\gamma]$ of $\bR$ generated by
$\gamma$.  Since $\gamma = 1+1/\gamma$ and $(1/\gamma)^2 = 1-
(1/\gamma)$, we note that $\bZ[\gamma] = \bZ[1/\gamma] = \bZ \oplus
\bZ\cdot (1/\gamma)$ is a free $\bZ$-module with basis
$\{1,1/\gamma\}$.  The formulas
\begin{equation}
 \label{formulas:gamma^-i}
 \gamma^{-i} = f(-i) + f(-i-1)/\gamma
             = (-1)^i \big( f(i-2) - f(i-1)/\gamma \big),
\end{equation}
which follow from a quick recurrence argument, show that, for any
point $\alpha = m+n/\gamma \in \bZ[\gamma]$, the conditions
\eqref{conditions:thm:comb_deg} are equivalent to
\begin{equation}
 \label{conditions:decomp_alpha}
 \alpha = \gamma^{-i_1} + \cdots + \gamma^{-i_s}
 \et
 f(i_1) + \cdots + f(i_s) \le d.
\end{equation}
For each $d\in \bN$, let $E_d$ denote the set of points
$\alpha\in\bZ[\gamma]$ for which these conditions admit a solution
in integers $0\le i_1\le\cdots\le i_s$ for some $s\in\bN$ and, for
each $\alpha\in E_d$, let $s_d(\alpha)$ denote the largest value of
$s$ for which such a solution exists.  Then, Theorem
\ref{thm:comb_deg} can be restated by saying that, for each pair of
integers $d,s\in\bN$, the number of elements $\alpha$ of $E_d$ with
$s_d(\alpha)=s$ is the integer $\chi_d(s)$ given by
\eqref{formula:chi_d}.  It is in this form that Theorem
\ref{thm:comb_deg} will be proved in Section \ref{sec:comb},
together with a more compact description of the sets $E_d$.


\subsection{A class of extremal numbers}
 \label{subsec:2.2}
We first recall that a
real number $\xi$ is extremal if and only if there exists an
unbounded sequence of points $\ux_k = (x_{k,0}, x_{k,1}, x_{k,2})$
in $\bZ^3$, indexed by integers $k\ge 1$ in $\bN^*$, and a constant
$c_1\ge 1$ such that, for each $k\ge 1$, the first coordinate
$x_{k,0}$ of $\ux_k$ is non-zero and we have
\begin{itemize}
\item[E1:]
 $c_1^{-1} |x_{k,0}|^\gamma \le |x_{k+1,0}| \le c_1
 |x_{k,0}|^\gamma$,
\item[E2:]
 $\max\big\{ |x_{k,0}\xi-x_{k,1}|,\, |x_{k,0}\xi^2-x_{k,2}| \big\}
 \le c_1 |x_{k,0}|^{-1}$,
\item[E3:]
 $1\le |x_{k,0}x_{k,2}-x_{k,1}^2| \le c_1$,
\item[E4:]
 $1\le |\det(\ux_k,\ux_{k+1},\ux_{k+2})| \le c_1$.
\end{itemize}
This follows from Theorem 5.1 of \cite{Rb} upon noting that the
condition E2 forces the maximum norm of $\ux_k$ to behave like
$|x_{k,0}|$.  As in \cite{Rb}, it is convenient to identify each
triple $\ux=(x_0,x_1,x_2)$ with coefficients in a commutative ring
with the symmetric matrix
\[
 \ux = \begin{pmatrix}  x_0 &x_1\\ x_1 &x_2 \end{pmatrix}.
\]
Then, the condition E3 reads simply as $1\le |\det(\ux_k)| \le c_1$.
We also recall that, for a given extremal real number $\xi$, the
corresponding sequence $(\ux_k)_{k\ge 1}$ is unique up to its first
terms and up to multiplication of its terms by non-zero rational
numbers with bounded numerators and denominators (see Proposition
4.1 of \cite{Rd}).  Moreover, for such $\xi$ and such a sequence
$(\ux_k)_{k\ge 1}$ viewed as symmetric matrices, Corollary 4.3 of
\cite{Rd} ensures the existence of a non-symmetric and
non-skew-symmetric $2\times 2$ matrix
\begin{equation}
 \label{formula:M}
 M = \begin{pmatrix} a_{1,1} &a_{1,2}\\ a_{2,1} &a_{2,2} \end{pmatrix}
\end{equation}
with integer coefficients such that, for each sufficiently large
$k\ge 1$, the matrix $\ux_{k+2}$ is a rational multiple of
$\ux_{k+1}M_{k+1}\ux_k$ where
\begin{equation}
 \label{formula:Mk}
 M_{k+1} =
   \begin{cases}
      M &\text{if $k$ is odd,}\\
     \tM &\text{if $k$ is even,}
   \end{cases}
\end{equation}
and where $\tM$ denotes the transpose of $M$.  The present paper
deals with a special class of extremal numbers.

\begin{hypothesis}
 \label{hyp}
In the sequel, we fix an extremal real number $\xi$, a corresponding
sequence $(\ux_k)_{k\ge 1}$ and a corresponding matrix $M$
satisfying the additional property that, for each $k\ge 1$, we have
\begin{equation}
 \label{formula:hyp}
  \det(\ux_k) = 1
  \et
  \ux_{k+2}=\ux_{k+1}M_{k+1}\ux_k.
\end{equation}
\end{hypothesis}

The first condition $\det(\ux_k)=1$ is restrictive as there exist
extremal real numbers with no corresponding sequence $(\ux_k)_{k\ge
1}$ in $\SL_2(\bZ)$, but it is not empty as it is fulfilled by any
real number whose continued fraction expansion is given by an
infinite Fibonacci word constructed on two non-commuting words in
positive integers, provided that both words have even length (see
\cite{Re}). The second condition however is no real additional
restriction. It is achieved by omitting the first terms of the
sequence $(\ux_k)_{k\ge 1}$ if necessary, by choosing $M$ so that
$\ux_3 = \ux_2 M \ux_1$ and then by multiplying recursively each
$\ux_k$ with $k\ge 4$ by $\pm 1$ so that the second equality in
\eqref{formula:hyp} holds for each $k\ge 1$.  In particular, we have
$\det(M)=1$. Since $\tM\neq \pm M$, this in turn implies that we
have $a_{1,1}\neq 0$ or $a_{2,2}\neq 0$.


\subsection{A subring of the ring of sequences}
 \label{subsec:2.3}
Let $S$ denote the ring of sequences of real numbers $(a_k)_{k\ge
1}$ indexed by the set $\bN^*$ of positive integers, with
component-wise addition and multiplication, and let $\gS$ denote the
quotient of $S$ by the ideal $S_0$ of sequences with finitely many
non-zero terms.  Two sequences $(a_k)_{k\ge 1}$ and $(b_k)_{k\ge 1}$
in $S$ thus represent the same element of $\gS$ if and only if
$a_k=b_k$ for each sufficiently large integer $k$.  We view $\bR$ as
a subring of $\gS$ by identifying each $x\in\bR$ with the image of
the constant sequence $(x)_{k\ge 1}$ modulo $S_0$.  This gives $\gS$
the structure of an $\bR$-algebra.  By restriction of scalars, we
may also view $\gS$ as a $\bQ$-algebra.  In particular, given
elements $\gA_1,\dots,\gA_\ell$ of $\gS$ we can form the
sub-$\bQ$-algebra $\bQ[\gA_1,\dots,\gA_\ell]$ that they generate.

For each $i\in\bZ$ and each $j=0,1,2$, we define an element
$\gX_j^{(i)}$ of $\gS$ by
\begin{equation}
 \label{formula:gX_j^i}
 \gX_j^{(i)} = \text{\ class of\ }\big( x_{2k+i,j} \big)_{k\ge 1}
               \text{\ in $\gS$,}
\end{equation}
where for definiteness we agree that $x_{2k+i,j}=0$ when $2k+i \le
0$, although the resulting element of $\gS$ is independent of this
choice.  Clearly, $\gX_j^{(i+2)}$ differs simply from $\gX_j^{(i)}$
by a shift, but nevertheless they are quite different from the
algebraic point of view. For each $i\in\bZ$, we also define a triple
\begin{equation}
 \label{formula:gX}
 \gX^{(i)}
  = \big( \gX_0^{(i)},\, \gX_1^{(i)},\, \gX_2^{(i)} \big)
    \in \gS^3.
\end{equation}
Identifying these triples with $2\times 2$ symmetric matrices
according to our general convention, and using \eqref{formula:Mk} to
extend the definition of $M_{k+1}$ to all integers $k$, Hypothesis
\ref{hyp} gives
\begin{equation}
 \label{formula:hyp_seq}
  \det(\gX^{(i)}) = 1
  \et
  \gX^{(i+2)} = \gX^{(i+1)} M_{i+1} \gX^{(i)}
\end{equation}
for each $i\in \bZ$. Our goal in this paper is to study the
sub-$\bQ$-algebra $\bQ[\gX^{(0)},\,\gX^{(-1)}]$ of $\gS$.  For this
purpose, we form the polynomial ring $\QX$ in six indeterminates
\[
 \uX = (X_0, X_1, X_2)
 \et
 \uX^* = (X_0^*, X_1^*, X_2^*).
\]
We first note that $\QgX$ contains the coordinates of $\gX^{(i)}$
for each $i\in\bZ$.

\begin{lemma}
 \label{lemma:bi-degree_gX}
For each $i\in\bZ$, the coordinates of $\gX^{(-i)}$ can be written
as values at the point $\big(\gX^{(0)},\,\gX^{(-1)}\big)$ of
polynomials of $\QX$ that are separately homogeneous of degree
$|f(i-2)|$ in $\uX$ and homogeneous of degree $|f(i-1)|$ in $\uX^*$.
\end{lemma}

\begin{proof}
By \eqref{formula:hyp_seq}, the coordinates of $\gX^{(i+2)}$ are
bilinear forms in $\big(\gX^{(i+1)},\,\gX^{(i)}\big)$ for each
$i\in\bZ$.  Since each matrix $\gX^{(i)}M_i$ has determinant $1$,
its inverse is its adjoint. Thus we find
\[
 \gX^{(-i-2)} = \Adj\big( \gX^{(-i-1)} M_{-i-1} \big)\, \gX^{(-i)},
\]
showing that, for each $i\in\bZ$, the coordinates of $\gX^{(-i-2)}$
are also bilinear in $\big(\gX^{(-i-1)},\,\gX^{(-i)}\big)$. By
recurrence, this implies that, for each $i\ge 0$, the coordinates of
$\gX^{(-i)}$ (resp.\ $\gX^{(i)}$) are the values at
$\big(\gX^{(0)},\, \gX^{(-1)}\big)$ of bi-homogeneous polynomials of
bi-degree $(f(i-2),f(i-1))$ (resp.\ $(f(i),f(i-1))$).  The
conclusion follows since, for the same values of $i$, the formulas
\eqref{formulas:gamma^-i} give $f(i)=|f(-i-2)|$ and
$f(i-1)=|f(-i-1)|$.
\end{proof}

The preceding result implicitly uses the natural surjective ring
homomorphism
\begin{equation}
 \label{def:pi}
 \begin{matrix}
 \pi\colon &\QX &\longrightarrow & \QgX \\[2pt]
 &P(\uX,\uX^*) &\longmapsto &P(\gX^{(0)},\gX^{(-1)}).
 \end{matrix}
\end{equation}
Our first goal is to describe the kernel of this map.

\begin{lemma}
 \label{lemma:I}
The kernel of $\pi$ contains the ideal $I$ of $\QX$ generated by the
polynomials
\begin{equation}
 \label{def:gen_I}
 \begin{gathered}
 \det(\uX) - 1 = X_0X_2 - X_1^2 - 1, \\
 \det(\uX^*) - 1 = X^*_0 X^*_2 - (X^*_1)^2 - 1, \\
 \Phi(\uX,\uX^*) =
   a_{1,1} \left| \begin{matrix}
            X^*_0 &X^*_1 \\ X_0 &X_1
            \end{matrix}
     \right|
 + a_{1,2} \left| \begin{matrix}
            X^*_1 &X^*_2 \\ X_0 &X_1
            \end{matrix}
     \right|
 + a_{2,1} \left| \begin{matrix}
            X^*_0 &X^*_1 \\ X_1 &X_2
            \end{matrix}
     \right|
 + a_{2,2} \left| \begin{matrix}
            X^*_1 &X^*_2 \\ X_1 &X_2
            \end{matrix}
     \right|.
 \end{gathered}
\end{equation}
\end{lemma}

\begin{proof} The first equality in \eqref{formula:hyp_seq} tells us
that $\det(\gX^{(i)})=1$ for each $i\in\bZ$.  Applying this with
$i=0$ and $i=-1$, we deduce that $\det(\uX)-1$ and $\det(\uX^*)-1$
belong to the kernel of $\pi$.  On the other hand, the second
equality in \eqref{formula:hyp_seq} gives $\gX^{(1)} = \gX^{(0)} M
\gX^{(-1)}$ and so
\[
 \begin{pmatrix}
  \gX_0^{(1)} &\gX_1^{(1)}\\
  \gX_1^{(1)} &\gX_2^{(1)}
 \end{pmatrix}
 =
 \begin{pmatrix}
  \gX_0^{(0)} &\gX_1^{(0)}\\
  \gX_1^{(0)} &\gX_2^{(0)}
 \end{pmatrix}
 \begin{pmatrix} a_{1,1} &a_{1,2}\\ a_{2,1} &a_{2,2} \end{pmatrix}
 \begin{pmatrix}
  \gX_0^{(-1)} &\gX_1^{(-1)}\\
  \gX_1^{(-1)} &\gX_2^{(-1)}
 \end{pmatrix}.
\]
In particular, the matrix product on the right gives rise to a
symmetric matrix.  This fact translates into
$\Phi(\gX^{(0)},\gX^{(-1)}) = 0$, and so $\Phi \in \ker(\pi)$.
\end{proof}

We will see below that $I$ is precisely the kernel of $\pi$.  In
Section \ref{sec:complement} we will provide an alternative proof of
this, suggested by Daniel Daigle, showing moreover that $I$ is a
prime ideal of $\QX$ of rank $3$, and therefore that $\QgX$ is an
integral domain of transcendence degree $3$ over $\bQ$.


\subsection{Asymptotic behaviors}
 \label{subsec:2.4}
The units of $\gS$ are the elements of $\gS$ which are
represented by sequences $(a_k)_{k\ge 1}$ with $a_k\neq 0$ for each
sufficiently large $k$.  We define an equivalence relation $\sim$ on
the group $\gS^*$ of units of $\gS$ by writing $\gA\sim\gB$ when
$\gA$ and $\gB$ are represented respectively by sequences
$(a_k)_{k\ge 1}$ and $(b_k)_{k\ge 1}$ with $\lim_{k\to \infty}
a_k/b_k =1$.  Then, for each $i\in\bZ$, the condition E2 implies
that
\begin{equation}
 \label{estimates:coord_gX}
 \gX_1^{(i)} \sim \xi\, \gX_0^{(i)}
 \et
 \gX_2^{(i)} \sim \xi^2\, \gX_0^{(i)}.
\end{equation}
For this reason, we regard the points $\gX^{(i)}$ as generic
(projective) approximations to the triple $(1,\xi,\xi^2)$ and, in
view of Lemma \ref{lemma:bi-degree_gX}, we say that $\QgX$ is the
ring of approximation triples to $(1,\xi,\xi^2)$. Since the second
formula in \eqref{formula:hyp_seq} gives
\[
 \gX_0^{(i+2)}
  =
 \begin{pmatrix} \gX_0^{(i+1)} &\gX_1^{(i+1)} \end{pmatrix}
 M_{i+1}
 \begin{pmatrix} \gX_0^{(i)}\\ \gX_1^{(i)} \end{pmatrix},
\]
we also find that
\begin{equation}
 \label{estimates:gX0}
 \gX_0^{(i+2)} \sim  \theta\, \gX_0^{(i+1)}\, \gX_0^{(i)}
\end{equation}
where
\begin{equation*}
 \theta
 =
 \begin{pmatrix} 1 &\xi \end{pmatrix}
 M_{i+1}
 \begin{pmatrix} 1\\ \xi \end{pmatrix}
 =
 a_{1,1} + (a_{1,2}+a_{2,1})\xi + a_{2,2}\xi^2
\end{equation*}
is independent of $i$ and non-zero because $\xi$ is transcendental
over $\bQ$ and $M$ is not skew-symmetric.

For each $d\in\bN$, we denote by $\QX_{\le d}$ the subspace of $\QX$
consisting of all polynomials of total degree at most $d$, and by
$\QgX_{\le d}$ its image under the evaluation map $\pi$.  The next
lemma provides a variety of elements of the latter set, with
explicit behavior.

\begin{lemma}
 \label{lemma:gM}
Let $d\in\bN$.  For each $\alpha = m+n/\gamma \in E_d$ and each
integer $j$ with $0\le j \le 2s_d(\alpha)$, there exists an element
$\gM_{\alpha,j}$ of $\QgX_{\le d}$ with
\[
 \gM_{\alpha,j}
 \sim
 \theta^{m+n-s_d(\alpha)} \xi^j
  \big( \gX_0^{(0)} \big)^m
  \big( \gX_0^{(-1)} \big)^n.
\]
\end{lemma}

\begin{proof}
By recurrence, we deduce from \eqref{estimates:gX0} that, for each
$i\in\bZ$, we have
\begin{equation*}
 \gX_0^{(i)}
 \sim
 \theta^{f(i+1)-1}
  \big( \gX_0^{(0)} \big)^{f(i)}
  \big( \gX_0^{(-1)}   \big)^{f(i-1)}.
\end{equation*}
Now, let $\alpha = m+n/\gamma \in E_d$, let $s=s_d(\alpha)$, and let
$j$ be an integer with $0\le j \le 2s$.  By definition, there exist
integers $0 \le i_1 \le \cdots \le i_s$ satisfying
\eqref{conditions:decomp_alpha}.  Choose also integers
$j_1,\dots,j_s\in \{0,1,2\}$ such that $j_1+\cdots+j_s=j$.  Then, we
find
\[
 \prod_{k=1}^s \gX_{j_k}^{(-i_k)}
 \sim
 \xi^j
 \prod_{k=1}^s \theta^{f(-{i_k}+1) - 1}\,
   \big( \gX_0^{(0)} \big)^{f(-i_k)}
   \big( \gX_0^{(-1)} \big)^{f(-i_k-1)}.
\]
Since \eqref{conditions:decomp_alpha} and
\eqref{conditions:thm:comb_deg} are equivalent, the product on the
right is simply $\theta^{m+n-s} \xi^j \big( \gX_0^{(0)} \big)^m
\big( \gX_0^{(-1)} \big)^n$.  The conclusion then follows by
observing that, according to Lemma \ref{lemma:bi-degree_gX}, the
product on the left is the value at $\big(\gX^{(0)},\,
\gX^{(-1)}\big)$ of some bi-homogeneous polynomial of $\QX$ with
bi-degree $\sum_{k=1}^s \big( f(i_k-2),f(i_k-1) \big)$, and thus
with total degree $\sum_{k=1}^s f(i_k) \le d$.
\end{proof}

In fact, we claim that the elements $\gM_{\alpha,j}$ constructed in
the preceding lemma form a basis of $\QgX_{\le d}$.  To prove this,
we will first show that they are linearly independent over $\bQ$,
and count them using Theorem \ref{thm:comb_deg}. This will provide a
lower bound for the dimension of $\QgX_{\le d}$. Next, using the
fact that the ideal $I$ is contained in the kernel of the evaluation
map $\pi$, we will find that the same number is also an upper bound
for this dimension. This will prove our claim and will bring other
consequences as well.  We now proceed to the first step of this
programme.


\subsection{Growth estimates}
 \label{subsec:2.5}
Let $\gS_+$ denote the subgroup of $\gS^*$ whose elements are
represented by sequences with positive terms. Given $\gA,\, \gB \in
\gS_+$, we write $\gA \ll \gB$ or $\gB \gg \gA$ if there exists a
constant $c>0$ such the corresponding sequences $(a_k)_{k\ge 1}$ and
$(b_k)_{k\ge 1}$ satisfy $a_k \le c b_k$ for each sufficiently large
index $k$.  We write $\gA\asymp \gB$ if we both have $\gA \ll \gB$
and $\gA \gg \gB$.  The latter is an equivalence relation on
$\gS_+$, and the condition E1 in \S\ref{subsec:2.2} can be expressed
in the form
\begin{equation*}
 \big| \gX_0^{(i+1)} \big| \asymp \big| \gX_0^{(i)} \big|^\gamma
\end{equation*}
for each $i\in\bZ$, where the absolute value and exponentiation are
taken ``component-wise''.  In particular, for any $d\in\bN$, $\alpha
= m+n/\gamma \in E_d$ and $j \in \{ 0, 1, \dots, 2s_d(\alpha)\}$,
the element $\gM_{\alpha,j}$ of $\QgX_{\le d}$ provided by Lemma
\ref{lemma:gM} satisfies
\[
 \big|\gM_{\alpha,j}\big|
 \asymp
 \big| \gX_0^{(0)} \big|^m
 \big| \gX_0^{(-1)} \big|^n
 \asymp
 \big| \gX_0^{(0)} \big|^{m+n/\gamma}
 =
 \big| \gX_0^{(0)} \big|^\alpha.
\]
Since each $E_d$ is a finite set of positive real numbers, this
leads to the following conclusion.

\begin{lemma}
 \label{lemma:comb_gM}
Let $d\in\bN$, let $r_{\alpha,j}$ ($\alpha\in E_d$, $0\le j \le
2s_d(\alpha)$) be rational numbers not all zero, let $\alpha' =
m+n/\gamma$ be the largest element of $E_d$ for which at least one
of the numbers $r_{\alpha',j}$ ($0\le j \le 2s_d(\alpha')$) is
non-zero, and put $s'=s_d(\alpha')$. Then, with the notation of
Lemma \ref{lemma:gM}, the linear combination $\gA = \sum_{\alpha\in
E_d} \sum_{j=0}^{2s_d(\alpha)} r_{\alpha,j}\, \gM_{\alpha,j}$
satisfies
\begin{equation*}
 \gA
 \ \sim\
 \bigg( \sum_{j=0}^{2s'} r_{\alpha',j}\, \xi^j \bigg)
  \theta^{m+n-s'}
  \big( \gX_0^{(0)} \big)^m
  \big( \gX_0^{(-1)} \big)^n
 \et
 |\gA| \asymp\
 \big| \gX_0^{(0)} \big|^{\alpha'}.
\end{equation*}
\end{lemma}

We are now ready to complete the first step of the programme
outlined at the end of the subsection \ref{subsec:2.4}.

\begin{lemma}
 \label{lemma:indep_gM}
Let $d\in\bN$. The elements $\gM_{\alpha,j}$ ($\alpha\in E_d$,
$j=0,\dots,2s_d(\alpha)$) constructed in Lemma \ref{lemma:gM} form a
$\bQ$-linearly independent subset of $\QgX_{\le d}$ with cardinality
$(4d^3+6d^2+8d+3)/3$.
\end{lemma}

\begin{proof}
Lemma \ref{lemma:comb_gM} shows that the elements $\gM_{\alpha,j}$
are linearly independent over $\bQ$.  By Theorem \ref{thm:comb_deg},
their number is
\[
 \sum_{\alpha\in E_d} (2s_d(\alpha)+1)
 = \sum_{s=0}^d \chi_d(s) (2s+1)
 = (d+1)(2d+1) + \sum_{s=0}^{d-1} (2s+1)^2.
\]
\end{proof}


\subsection{Computation of an Hilbert function}
 \label{subsec:2.6}
We introduce a new variable $U$, make the ring $\bQ[\uX,\uX^*,U]$
into a graded ring for the total degree, and denote by $I_1$ the
homogeneous ideal of this ring generated by
\begin{equation}
 \label{generators:I1}
 \det(\uX)-U^2,\quad
 \det(\uX^*)-U^2,\quad
 \Phi(\uX,\uX^*).
\end{equation}
Then, for each $d\in \bN$, Lemma \ref{lemma:I} ensures that we have
a surjective $\bQ$-linear map
\begin{equation}
 \label{iso:I1}
 \begin{aligned}
   \big( \bQ[\uX,\uX^*,U]/I_1 \big)_d
      &\longrightarrow \QgX_{\le d}\\
    P(\uX,\uX^*,U) + I_1
      &\longmapsto P\big(\gX^{(0)},\, \gX^{(-1)},\, 1\big).
 \end{aligned}
\end{equation}
In particular, this gives
\begin{equation}
 \label{upper_bound_dim_QgX}
 \dim_\bQ \QgX_{\le d}
 \le H(I_1;d)
 := \dim_\bQ \big( \bQ[\uX,\uX^*,U]/I_1 \big)_d.
\end{equation}
Thus, in order to complete the programme outlined at the end of the
subsection \ref{subsec:2.4}, it remains to compute the Hilbert
function $H(I_1;d)$ of $I_1$.  We achieve this by showing first that
the generators \eqref{generators:I1} of $I_1$ form a regular
sequence in $\bQ[\uX,\uX^*,U]$.

Recall that a \emph{regular sequence} in a ring $R$ is a finite
sequence of elements $a_1,\dots,a_n$ of $R$ such that, for
$i=1,\dots,n$, the multiplication by $a_i$ in $R/(a_1, \dots,
a_{i-1})$ is injective (with the convention that $(a_1, \dots,
a_{i-1}) = (0)$ for $i=0$). If $R$ is a polynomial ring in $m$
variables over a field, then, for any integer $n$ with $1\le n\le
m$, a sequence of $n$ homogeneous polynomials $a_1,\dots,a_n$ of $R$
is regular if and only if the ideal $(a_1,\dots,a_n)$ that it
generates has rank (or codimension) equal to $n$.  In that case, any
permutation of $a_1,\dots,a_n$ is a regular sequence.

\begin{lemma}
 \label{lemma:reg_seq_det}
The polynomials $\det(\uX)$, $\det(\uX^*)$ and $\Phi(\uX,\uX^*)$
form a regular sequence in $\QX$.
\end{lemma}

\begin{proof}
Put $R:=\QX$.  Since $\det(\uX)$ and $\det(\uX^*)$ are relatively
prime, they form a regular sequence in $R$.  Moreover, the ideal
that they generate is the kernel of the endomorphism of $R$ which
maps $X_i$ to $X_0^{2-i}X_1^i$ and $X^*_i$ to $(X^*_0)^{2-i}
(X^*_1)^i$ for $i=0,1,2$.  The conclusion follows by observing that
the image of $\Phi(\uX,\uX^*)$ under this map is
\[
 (X_1X^*_0-X_0X_1^*)
 (a_{1,1}X_0X^*_0+a_{1,2}X_0X^*_1+a_{2,1}X_1X_0^*+a_{2,2}X_1X_1^*),
\]
which is a non-zero polynomial.
\end{proof}

We can now turn to the ideal $I_1$.

\begin{lemma}
 \label{lemma:reg_seq_I1}
The generators \eqref{generators:I1} of $I_1$ form a regular
sequence in $\bQ[\uX,\uX^*,U]$. For each $d\in\bN$, we have
\[
 H(I_1;d)
 =
 (4d^3+6d^2+8d+3)/3.
\]
\end{lemma}

\begin{proof}
Since the natural isomorphism $\bQ[\uX,\uX^*,U]/(U) \to \QX$ induced
by the specialization $U\mapsto 0$ maps the sequence of polynomials
\eqref{generators:I1} to the regular sequence of $R$ studied in
Lemma \ref{lemma:reg_seq_det}, we deduce that $U$, $\det(\uX)-U^2$,
$\det(\uX^*)-U^2$ and $\Phi(\uX,\uX^*)$ form a regular sequence in
$\bQ[\uX,\uX^*,U]$.  Since these are homogeneous polynomials, it
follows that the last three of them, which generate $I_1$, form a
regular sequence. Since the latter are homogeneous of degree $2$ and
since $\bQ[\uX,\uX^*,U]$ is a polynomial ring in $7$ variables, the
Hilbert series of the ideal $I_1$ is given by
\[
 \sum_{d=0}^\infty H(I_1;d) T^d
 = \frac{(1-T^2)^3}{(1-T)^7}
 = \frac{(1+T)^3}{(1-T)^4}
 = (1+T)^3 \sum_{d=0}^\infty \binom{d+3}{3} T^d,
\]
and a short computation completes the proof.
\end{proof}


\subsection{Conclusion}
 \label{subsec:2.7}
Combining the above result with \eqref{upper_bound_dim_QgX} and
Lemma \ref{lemma:indep_gM}, we obtain finally:

\begin{theorem}
 \label{thm:basis_gM}
Let $d\in\bN$. Then, the map \eqref{iso:I1} is an isomorphism of
vector spaces over $\bQ$, and the elements $\gM_{\alpha,j}$
($\alpha\in E_d$, $j=0,\dots,2s_d(\alpha)$) constructed in Lemma
\ref{lemma:gM} form a basis of $\QgX_{\le d}$. The dimension of the
latter vector space is $(4d^3+6d^2+8d+3)/3$.
\end{theorem}

Applying first Lemma \ref{lemma:comb_gM} and then the growth
estimates of Subsection \ref{subsec:2.5}, we deduce from this the
following two consequences.

\begin{corollary}
 \label{cor1:thm:basis}
Let $d\in\bN$.  For non-zero element $\gA$ of $\QgX_{\le d}$, there
exists a point $\alpha =  m+n/\gamma \in E_d$ and a polynomial
$A\in\bQ[T]$ of degree at most $2s_d(\alpha)$ such that
\[
 \gA
 \sim
 \theta^{m+n-s_d(\alpha)} A(\xi)
 \big( \gX_0^{(0)} \big)^m\,
 \big( \gX_0^{(-1)} \big)^n.
\]
\end{corollary}

\begin{corollary}
 \label{cor2:thm:basis}
For non-zero element $\gA$ of $\QgX$, there exists a point $\alpha
\in \bZ[\gamma]$ such that
\[
 \gA
 \asymp
 \big| \gX_0^{(0)} \big|^\alpha.
\]
The map $\gA\mapsto \alpha$ is a rank two valuation on the ring
$\QgX$.
\end{corollary}

Finally, we note that, for each $d\in\bN$, the linear map
\eqref{iso:I1} factors through the map from $(\QX_{\le d}+I)/I$ to
$\QgX_{\le d}$ induced by $\pi$.  Since the former is an
isomorphism, the latter is also an isomorphism, and so:

\begin{corollary}
 \label{cor:ker_pi}
The ideal $I$ defined in Lemma \ref{lemma:I} is the kernel of the
evaluation map $\pi$ from $\QX$ to $\QgX$.
\end{corollary}

%
%

\section{Analogous results in bi-degree}
 \label{sec:bideg}

The following result is analogous to Theorem \ref{thm:comb_deg}.

\begin{theorem}
 \label{thm:comb_bideg}
Let $\ud = (d_1,d_2)\in\bN^2$.  For each $s\in\bN$, let $\chi_\ud(s)
= \chi_{d_1,d_2}(s)$ denote the number of points $(m,n)\in\bZ^2$ for
which the conditions
\begin{equation}
 \label{conditions:thm:comb_bideg}
 (m,n) = \sum_{k=1}^s \big( f(-i_k),\, f(-i_k-1) \big),
 \quad
 \sum_{k=1}^s f(i_k-2) \le d_1
 \et
 \sum_{k=1}^s f(i_k-1) \le d_2,
\end{equation}
admit a solution integers $0\le i_1\le\cdots\le i_s$, and $s$ is
maximal with this property.  Then, 
\begin{equation}
 \label{formula:chi_ud}
  \chi_{d_1,d_2}(s) =
  \begin{cases}
    2\min\{ d_1,\, d_2,\, s,\, d_1+d_2-s\} + 1
      &\text{if\, $0\le s\le d_1+d_2$,}\\
    0 &\text{if\, $s>d_1+d_2$.}
  \end{cases}
\end{equation}
\end{theorem}

Note that this function $\chi_{d_1,d_2}(s)$ possess several
symmetries. For each $(d_1,d_2)\in\bN^2$ and each
$s=0,\dots,d_1+d_2$, it satisfies
\begin{equation*}
 \chi_{d_1,d_2}(s) = \chi_{d_2,d_1}(s)
 \et
 \chi_{d_1,d_2}(s) = \chi_{d_1,d_2}(d_1+d_2-s).
\end{equation*}
As in Section \ref{sec:ring}, we note that, for a point $\alpha =
m+n/\gamma \in \bZ[\gamma]$, the conditions
\eqref{conditions:thm:comb_bideg} are equivalent to
\begin{equation}
 \label{conditions:decomp_alpha_bideg}
 \alpha = \sum_{k=1}^s \gamma^{-i_k},
 \quad
 \sum_{k=1}^s f(i_k-2) \le d_1
 \et
 \sum_{k=1}^s f(i_k-1) \le d_2.
\end{equation}
For each $\ud\in \bN^2$, we denote by $E_\ud$ the set of $\alpha \in
\bZ[\gamma]$ for which these conditions admit a solution in integers
$0\le i_1\le\cdots\le i_s$ for some $s\in\bN$ and, for each
$\alpha\in E_\ud$, we denote by $s_\ud(\alpha)$ the largest such
$s$. Then, Theorem \ref{thm:comb_bideg} tells us that, for given
$\ud\in\bN^2$ and $s\in\bN$, the number of elements $\alpha$ of
$E_\ud$ with $s_\ud(\alpha)=s$ is $\chi_\ud(s)$ given by
\eqref{formula:chi_ud}. This will be proved in Section
\ref{sec:comb}.

For each $\ud = (d_1,d_2) \in\bN^2$, we also denote by $\QX_{\le
\ud}$ the set of polynomials of $\QX$ with degree at most $d_1$ in
$\uX$ and degree at most $d_2$ in $\uX^*$. We also write $\QgX_{\le
\ud}$ for the image of that set under the evaluation map $\pi$
defined by \eqref{def:pi}. We can now state and prove the following
bi-degree analog of Theorem \ref{thm:basis_gM}.

\begin{theorem}
 \label{thm:basis_gM'}
Let $\ud=(d_1,d_2)\in\bN^2$. For each $\alpha = m+n/\gamma \in
E_\ud$ and each $j=0,\dots,2s_\ud(\alpha)$, there exists an element
$\gM'_{\alpha,j}$ of $\QgX_{\le \ud}$ with
\[
 \gM'_{\alpha,j}
 \sim
 \theta^{m+n-s_\ud(\alpha)} \xi^j
  \big( \gX_0^{(0)} \big)^m
  \big( \gX_0^{(-1)} \big)^n.
\]
Any such choice of elements, one for each pair $(\alpha,j)$,
provides a basis of $\QgX_{\le \ud}$.  This vector space has
dimension $(d_1+d_2+1)(2d_1d_2+d_1+d_2+1)$.
\end{theorem}

\begin{proof}
The existence of the elements $\gM'_{\alpha,j}$ is established
exactly as in the proof of Lemma \ref{lemma:gM}, upon replacing
everywhere the symbol $d$ by $\ud$, using
\eqref{conditions:decomp_alpha_bideg} and
\eqref{conditions:thm:comb_bideg} instead of
\eqref{conditions:decomp_alpha} and \eqref{conditions:thm:comb_deg}.
Fix such a choice of elements. The fact that they are linearly
independent over $\bQ$ is proved as in \S\ref{subsec:2.5}, upon
observing that the statement of Lemma \ref{lemma:comb_gM} still
holds when $d$ is replaced by $\ud$ and $\gM_{\alpha,j}$ by
$\gM'_{\alpha,j}$. According to Theorem \ref{thm:comb_bideg}, they
form a set of cardinality $\sum_{s=0}^{d_1+d_2} \chi_\ud(s) (2s+1)$.
Since $\chi_\ud(s) = \chi_\ud(d_1+d_2-s)$ for $s=0,\dots,d_1+d_2$,
this cardinality is also given by
\[
 \frac{1}{2} \sum_{s=0}^{d_1+d_2} \Big( \chi_\ud(s) (2s+1)
 + \chi_\ud(s)  (2(d_1+d_2-s)+1) \Big)
 =
 (d_1+d_2+1) \sum_{s=0}^{d_1+d_2}\chi_\ud(s).
\]
A short computation based on the formula \eqref{formula:chi_ud}
shows that the right-most sum is equal to $2d_1d_2+d_1+d_2+1$ (an
alternative approach is to note that this sum is the cardinality of
$E_\ud$ and to use Corollary \ref{cor:card_Ed}).  Thus the elements
$\gM'_{\alpha,j}$ span a subspace of $\QgX_{\le \ud}$ of dimension
$(d_1+d_2+1)(2d_1d_2+d_1+d_2+1)$.  To complete the proof, it remains
only to show that the dimension of $\QgX_{\le \ud}$ is no more than
this.  To that end, we proceed as in \S \ref{subsec:2.6}. We
introduce two new indeterminates $V$ and $V^*$ and, for each
$\un=(n_1,n_2)\in\bN^2$, we denote by $\bQ[\uX,V,\uX^*,V^*]_\un$ the
subspace of $\bQ[\uX,V,\uX^*,V^*]$ whose elements are homogeneous in
$(\uX,V)$ of degree $n_1$ and homogeneous in $(\uX^*,V^*)$ of degree
$n_2$. This makes the polynomial ring $R_2:=\bQ[\uX,V,\uX^*,V^*]$
into a $\bN^2$-graded ring.  Let $I_2$ denote the bi-homogeneous
ideal of $R_2$ generated by
\begin{equation}
 \label{generators:I2}
 \det(\uX)-V^2,\quad
 \det(\uX^*)-(V^*)^2,\quad
 \Phi(\uX,\uX^*).
\end{equation}
Lemma \ref{lemma:I} ensures that we have a surjective $\bQ$-linear
map in each bi-degree $\un$
\begin{equation}
 \label{iso:I2}
 \begin{aligned}
   \big( \bQ[\uX,V,\uX^*,V^*]/I_2 \big)_\un
      &\longrightarrow \QgX_{\le \un}\\
    P(\uX,V,\uX^*,V^*) + I_1
      &\longmapsto P\big(\gX^{(0)},\, 1,\, \gX^{(-1)},\, 1\big).
 \end{aligned}
\end{equation}
In particular, this gives $\dim_\bQ \QgX_{\le \ud} \le H(I_2;\ud)$
where $H(I_2;\un)$ stands for the Hilbert function of $I_2$ at
$\un$, namely the dimension of the domain of the linear map
\eqref{iso:I2}. As in the proof of Lemma \ref{lemma:reg_seq_I1}, we
deduce from Lemma \ref{lemma:reg_seq_det} that the generators
\eqref{generators:I2} of $I_2$ form a regular sequence in $R_2$.
Since these generators are bi-homogeneous of bi-degree $(2,0)$,
$(0,2)$ and $(1,1)$, and since the grading of $R_2$ involves two
sets of $4$ variables, we deduce that the Hilbert series of $I_2$ is
\[
 \begin{aligned}
 \sum_{n_1,n_2\in\bN} H(I_2;n_1,n_2) T_1^{n_1} T_2^{n_2}
 &=
 \frac{(1-T_1^2)(1-T_2^2)(1-T_1T_2)}{(1-T_1)^4(1-T_2)^4}\\
 &=
 (1-T_1T_2) \sum_{n_1,n_2\in\bN}
   (n_1+1)^2(n_2+1)^2 T_1^{n_1}T_2^{n_2}.
 \end{aligned}
\]
This completes the proof as it implies that
\[
 H(I_2;\ud)
 = (d_1+1)^2(d_2+1)^2 - d_1^2d_2^2
 = (d_1+d_2+1)(2d_1d_2+d_1+d_2+1).
\]
\end{proof}

Note that this result implies that the statement of Corollary
\ref{cor1:thm:basis} still holds in bi-degree, with $d$ replaced by
$\ud$.

%
%

\section{Combinatorial study}
 \label{sec:comb}

This section is devoted to the proof of Theorems \ref{thm:comb_deg}
and \ref{thm:comb_bideg}.  As mentioned in \S\ref{subsec:2.1}, we
work within the ring $\bZ[\gamma]$.  We define
\[
 E= \{ \alpha\in\bZ[\gamma] \,;\, \alpha\ge 0 \}
 \et
 E^* = E \setminus \{0\},
\]
and note that, since $\gamma>0$, the sets $E_d$ and $E_\ud$ defined
respectively in Sections \ref{subsec:2.1} and \ref{sec:bideg} are
subsets of $E$.  Our first goal is to provide a more explicit
description of these.


\subsection{A partition}
 \label{subsec:4.1}
We first establish a partition of $E^*$.

\begin{proposition}
 \label{prop:partitionE}
The sets
\begin{align*}
 E^{(+)} &= \{ m+ n\gamma^{-1} \,;\, m,n\ge 1 \} \\
 \text{and}\quad
 E^{(i)} &= \{ m\gamma^{-i}+n\gamma^{-i-2} \,;\, m\ge 1, n\ge 0 \}
 \quad \text{for}\quad i\ge 0
\end{align*}
form a partition $E^* = E^{(+)} \coprod \left( \coprod_{i=0}^\infty
E^{(i)} \right)$ of $E^*$.
\end{proposition}

\begin{proof} Consider the bijection $\varphi\colon\bZ[\gamma]\to
\bZ^2$ which maps a point $m+n/\gamma$ to its pair of coordinates
$(m,n)$ relative to the basis $\{1,1/\gamma\}$ of $\bZ^2$.  It
identifies $E^*$ with the set of non-zero points $(m,n)$ of $\bZ^2$
whose argument in polar coordinates satisfies
\[
 -\arctan(\gamma) < \arg(m,n) < \pi-\arctan(\gamma).
\]
Using the formulas \eqref{formulas:gamma^-i}, a quick recurrence
argument shows that, for each index $i\ge 0$, the determinant of the
points $\varphi(\gamma^{-i})$ and $\varphi(\gamma^{-i-2})$ is
\[
 \left|
  \begin{matrix}
   f(i-2) &-f(i-1)\\ f(i) &-f(i+1)
  \end{matrix}
 \right|
 =
 (-1)^{i+1}.
\]
This means that $\{ \varphi(\gamma^{-i}), \varphi(\gamma^{-i-2})\}$
forms a basis of $\bZ^2$ for each $i\ge 0$. Since the points
$\varphi(\gamma^{-2i})$ have positive first coordinate, it also
means that $\arg\varphi(\gamma^{-2i})$ is a strictly decreasing
function of $i\ge 0$ starting from $\arg\varphi(\gamma^0)=0$.
Finally, since the points $\varphi(\gamma^{-2i-1})$ have positive
second coordinate, it tells us that $\arg\varphi(\gamma^{-2i-1})$ is
a strictly increasing function of $i\ge 0$ starting from
$\arg\varphi(\gamma^{-1})=\pi/2$.  In other words, we have
\[
 \cdots
 < \arg\varphi(\gamma^{-4})
 < \arg\varphi(\gamma^{-2})
 < \arg\varphi(\gamma^{0})
 < \arg\varphi(\gamma^{-1})
 < \arg\varphi(\gamma^{-3})
 < \arg\varphi(\gamma^{-5})
 < \cdots
\]
We conclude from this that a point $\alpha$ of $E^*$ belongs to
$E^{(+)}$ if and only if $\arg\varphi(\gamma^{0}) <
\arg\varphi(\alpha) < \arg\varphi(\gamma^{-1})$, and that it belongs
to $E^{(i)}$ for some $i\ge 0$ if and only if $\arg\varphi(\alpha)$
lies between $\arg\varphi(\gamma^{-i})$ and
$\arg\varphi(\gamma^{-i-2})$, with the first end point included and
the second excluded.  In particular the sets $E^{(+)}$ and $E^{(i)}$
with $i\ge 0$ are all disjoint.  They cover $E^*$ because the fact
that $\lim_{j\to \infty} f(j)/f(j-1) = \gamma$ implies that
$\arg\varphi(\gamma^{-2i})$ and $\arg\varphi(\gamma^{-2i-1})$ tend
respectively to $-\arctan(\gamma)$ and $\pi-\arctan(\gamma)$ as
$i\to\infty$.
\end{proof}

With our convention that an empty sum is zero, this implies that:

\begin{corollary}
 \label{cor:representation_alpha}
Any $\alpha\in E$ can be written in the form
\begin{equation}
 \label{dec:alpha:bis}
 \alpha = \gamma^{-i_1} + \cdots + \gamma^{-i_s}
\end{equation}
for a choice of integers $s\ge 0$ and $0\le i_1\le \cdots \le i_s$.
\end{corollary}

We say that a finite non-decreasing sequence of non-negative
integers $\ui=(i_1,\dots,i_s)$ is a \emph{representation} of a point
$\alpha$ of $E$ if it satisfies the condition \eqref{dec:alpha:bis}.
In particular, the only representation of the point $0$ is the empty
sequence.


\subsection{Degree and bi-degree}
For any finite non-decreasing sequence of non-negative integers $\ui
= (i_1, \dots, i_s)$, we define
\begin{equation*}
 \begin{gathered}
 d(\ui) = \sum_{k=1}^s f(i_k), \quad
 d_1(\ui) = \sum_{k=1}^s f(i_k-2), \quad
 d_2(\ui) = \sum_{k=1}^s f(i_k-1), \\
 \ud(\ui) = (d_1(\ui),d_2(\ui))
 \et
 s(\ui) = s.
 \end{gathered}
\end{equation*}
We say that $d(\ui)$, $\ud(\ui)$ and $s(\ui)$ are respectively the
\emph{degree}, \emph{bi-degree} and \emph{size} of the point $\ui$,
while $d_1(\ui)$ and $d_2(\ui)$ are respectively the \emph{first}
and \emph{second partial degrees} of $\ui$.  For the empty sequence,
all these integers are zero.  We also put a partial order on $\bN^2$
by writing $(m,n)\le (m',n')$ if $m\le m'$ and $n\le n'$.  We can
now state and prove:

\begin{proposition}
 \label{prop:degree_alpha}
Let $\alpha = m + n/\gamma \in E^*$ and let $\ui=(i_1,\dots,i_s)$ be
a representation of $\alpha$.  Then we have $d(\ui) \ge |m|+|n|$ and
$\ud(\ui) \ge (|m|,|n|)$.  Both inequalities are equalities if
$i_s\le 1$ or if $i_1, \dots, i_s$ share the same parity. Otherwise,
they become strict inequalities.  Moreover, we have $d_2(\ui) > |n|$
if $i_1,\dots,i_s$ contains a pair of positive integers not of the
same parity.
\end{proposition}

\begin{proof}
Since $\alpha=\gamma^{-i_1} + \cdots + \gamma^{-i_s}$, the formulas
\eqref{formulas:gamma^-i} imply that
\[
 m = \sum_{k=1}^s (-1)^{i_k} f(i_k-2)
 \et
 n = \sum_{k=1}^s (-1)^{i_k+1} f(i_k-1)
\]
{}From this we deduce that
\[
 |m| \le \sum_{k=1}^s f(i_k-2) = d_1(\ui)
 \et
 |n| \le \sum_{k=1}^s f(i_k-1) = d_2(\ui),
\]
and the conclusion follows because $f(-2)=1$, $f(-1)=0$ and $f(i)
\ge 1$ for each $i\in\bN$.
\end{proof}

Since, by Proposition \ref{prop:partitionE}, each $\alpha\in E^*$
admits a representation $\ui=(i_1,\dots,i_s)$ with $i_s\le 1$ or
with $i_1,\dots,i_s$ of the same parity, we deduce that

\begin{corollary}
 \label{cor:deg/bideg}
Each $\alpha = m + n/\gamma \in E$ admits a representation with
largest degree $d(\alpha):=|m|+|n|$ and largest bi-degree
$\ud(\alpha):= (|m|,|n|)$.
\end{corollary}

We say that the integers $d(\alpha)$ and $\ud(\alpha)$ defined in
the above corollary are respectively the \emph{degree} and
\emph{bi-degree} of $\alpha$.

Let $d\in\bN$ and $\ud=(d_1,d_2)\in\bN^2$.  In \S\ref{subsec:2.1}
(resp.\ \S\ref{sec:bideg}), we defined $E_d$ (resp.\ $E_\ud$) as the
set of points which admit a representation of degree $\le d$ (resp.\
of bi-degree $\le \ud$).  According to the corollary, it can also be
described as the set of elements of $E$ with degree $\le d$ (resp.\
with bi-degree $\le \ud$):
\[
 E_d = \{ m+n/\gamma\in E \,;\, |m|+|n| \le d \}
 \et
 E_\ud = \{ m+n/\gamma\in E \,;\, |m|\le d_1, \ |n| \le d_2 \}.
\]
We can now easily compute the cardinality of these sets.

\begin{corollary}
 \label{cor:card_Ed}
Let $d\in\bN$ and $\ud=(d_1,d_2)\in\bN^2$.  Then, we have $|E_d| =
d^2+d+1$ and $|E_\ud| = 2d_1d_2+d_1+d_2+1$.
\end{corollary}

\begin{proof}
Denote by $\cL$ the set of all non-zero points $(m,n)$ in $\bZ^2$
satisfying $|m|+|n| \le d$ (resp.\ $|m|\le d_1$ and $|n|\le d_2$).
Define also $\cL^+$ to be the set of points $(m,n)$ in $\cL$ for
which $m+n/\gamma > 0$.  Then, $E_d\setminus\{0\}$ (resp.\
$E_\ud\setminus\{0\}$) is in bijection with $\cL^+$.  As the sets
$\cL^+$ and $-\cL^+$ form a partition of $\cL$ in two subsets of the
same cardinality, it follows that the cardinality of $E_d$ (resp.\
of $E_\ud$) is $1 + |\cL|/2$, and the conclusion follows upon noting
that $|\cL|$ is $2d(d+1)$ (resp.\ $(2d_1+1)(2d_2+1)-1$).
\end{proof}


\subsection{Representations by quads}
 \label{subsec:3.3}
Let $\alpha\in E$. For each $d\in\bN$ such that $\alpha\in E_d$, we
define the \emph{size $s_d(\alpha)$ of $\alpha$ relative to $d$} to
be the largest size of a representation of $\alpha$ of degree $\le
d$ (see \S\ref{subsec:2.1}).  Similarly, for each $\ud\in\bN^2$ such
that $\alpha\in E_\ud$, we define the \emph{size $s_\ud(\alpha)$ of
$\alpha$ relative to $\ud$} to be the largest size of a
representation of $\alpha$ of bi-degree $\le \ud$ (see
\S\ref{sec:bideg}).  The next proposition shows that, in order to
compute the various degrees and sizes of $\alpha$, it suffices to
consider only representations of the form
\begin{equation}
 \label{dec:quad}
 \alpha = a\gamma^{-i} + b\gamma^{-i-1} + c\gamma^{-i-2}
\end{equation}
with $i,a,b,c\in \bN$, and $a\ge 1$ if $\alpha\neq 0$.

\begin{proposition}
 \label{prop:quad}
Let $d \in \bN^*$ and let $\ud \in \bN^2\setminus\{(0,0)\}$.  Each
$\alpha \in E_d\setminus\{0\}$ admits a representation $\ui =
(i_1,\dots,i_s)$ with degree $d(\ui) \le d$ and size $s =
s_d(\alpha)$ for which $i_s \le i_1+2$.  Similarly, each $\alpha \in
E_\ud\setminus\{0\}$ admits a representation $\ui = (i_1,\dots,i_s)$
with bi-degree $\ud(\ui) \le \ud$ and size $s = s_\ud(\alpha)$ for
which $i_s \le i_1+2$.
\end{proposition}

\begin{proof}
Let $\alpha \in E_d\setminus\{0\}$.  Put $s=s_d(\alpha)$, and choose
a representation $\ui=(i_1,\dots,i_s)$ of $\alpha$ of size $s$ with
minimal degree.  We claim that $\ui$ has all the required
properties.  First it satisfies $d(\ui)\le d$ by definition of
$s_d(\alpha)$.  It remains to show that $i_s\le i_1+2$.

To show this, we first observe that, for any pair of integers
$(p,k)$ with $k\ge 1$, we have
\begin{align*}
  \gamma^{-p} + \gamma^{-p-2k-1}
  &= \Big( \sum_{i=0}^{k-1} \gamma^{-p-2i-1} \Big)
     + \gamma^{-p-2k+1} \\
  \gamma^{-p} + \gamma^{-p-2k-2}
  &= \gamma^{-p-2}
     + \Big( \sum_{i=1}^{k} \gamma^{-p-2i} \Big)
     + \gamma^{-p-2k}.
\end{align*}
Assuming that $i_s\ge i_1+3$, these formulas show that the point
$\beta = \gamma^{-i_1} + \gamma^{-i_s}$ admits a representation $\uj
= (j_1,\dots,j_t)$ with coordinates of the same parity as $i_s$,
size $t=2$ if $i_s = i_1+3$, and size $t\ge 3$ if $i_s > i_1+3$.  In
this case, Proposition \ref{prop:degree_alpha} gives $d(\uj) =
d(\beta)$ and also $d(\beta) \le d(i_1,i_s)$ with the strict
inequality if $i_s=i_1+3$. Then, upon reorganizing terms in the
decomposition
\[
 \alpha = (\gamma^{-j_1}+\cdots+\gamma^{-j_t})
          + (\gamma^{-i_2}+\cdots+\gamma^{-i_{s-1}}),
\]
we get a representation $\ui'$ of $\alpha$ with degree $d(\ui') =
d(\ui) + d(\beta) - d(i_1,i_s)$ and size $s(\ui')=s+t-2$.  If $i_s =
i_1+3$, we have $d(\ui') < d(\ui)$ and $s(\ui')=s$ in contradiction
with the choice of $\ui$.  If $i_s> i_1+3$, we find that $d(\ui')
\le d(\ui) \le d$ and $s(\ui') > s = s_d(\alpha)$ in contradiction
with the definition of $s_d(\alpha)$.  Thus, we must have $i_s\le
i_1+2$.

This proves the first assertion of the proposition. The proof of the
second assertion is the same provided that one replaces everywhere
the word ``degree'' by ``bi-degree'', and the symbol $d$ by $\ud$.
\end{proof}

We define a \emph{quad} $q$ to be an expression of the form $q =
(i\,;\,a,b,c)$ with $i,a,b,c\in\bN$ and $a\ge 1$.  We say that a
quad $q$ as above represents a point $\alpha\in E$ if it satisfies
\eqref{dec:quad}.  Identifying it with the sequence formed by $a$
occurrences of $i$ followed by $b$ occurrences of $i+1$ and $c$
occurrences of $i+2$, the various notions of degree and size
translate to
\begin{equation}
 \label{formulas:degree&bidegree}
 \begin{gathered}
   d_1(q) = af(i-2)+bf(i-1)+cf(i), \quad
   d_2(q) = af(i-1)+bf(i)+cf(i+1), \\
   d(q) = d_1(q)+d_2(q), \quad
   \ud(q) = \big( d_1(q),\ d_2(q) \big)
   \et
   s(q) = a+b+c.
\end{gathered}
\end{equation}
In this context, Proposition \ref{prop:quad} shows that for any
$\alpha\in E^*$ and any integer $d\ge d(\alpha)$ (resp.\ any integer
pair $\ud \ge \ud(\alpha)$), the integer $s_d(\alpha)$ (resp.\
$s_\ud(\alpha)$) is the largest size of a quad of degree $\le d$
(resp.\ of bi-degree $\le \ud$) which represents $\alpha$.


\subsection{Sequences of quads}
 \label{subsec:4.4}
For each $\alpha\in E^*$, we denote by $Q_\alpha$ the set of quads
which represent $\alpha$, and, for each $\ud\in\bN^2$, we denote by
$Q_\ud$ the set of quads of bi-degree $\ud$.  Although we use the
same letter for both kinds of sets, the nature of the subscript
should in practice remove any ambiguity.  As we will see these
families have similar properties.  We start with those of the first
kind.

\begin{proposition}
 \label{prop:Qalpha}
Let $\alpha\in E^*$.  The set $Q_\alpha$ of all quads representing
$\alpha$ is an infinite set whose elements have distinct size. If we
order its elements by increasing size, then their sizes form an
increasing sequence of consecutive integers while their degrees,
bi-degrees, and second partial degrees form strictly increasing
sequences in $\bN$, $\bN^2$ and $\bN$ respectively. The element of
$Q_\alpha$ of smallest size is the quad of degree $d(\alpha)$ and
bi-degree $\ud(\alpha)$ associated to the representation of $\alpha$
given by Proposition \ref{prop:partitionE}.
\end{proposition}

\begin{proof}
The relation $\gamma^{-i} = \gamma^{-i-1}+\gamma^{-i-2}$ shows that,
for each $q = (i\,;\,a,b,c) \in Q_\alpha$, the quad
\begin{equation}
 \label{def:theta}
 \theta(q)
 = \begin{cases}
    (i\,;\,a-1,b+1,c+1) &\text{if $a\ge 2$,}\\
    (i+1\,;\,b+1,c+1,0) &\text{if $a = 1$}
   \end{cases}
\end{equation}
also represents $\alpha$.  This defines an injective map $\theta$
from $Q_\alpha$ to itself, which increases the size of a quad by
$1$.  Since the size of any quad is finite and non-negative, this
implies that any $q\in Q_\alpha$ can be written in a unique way in
the form $q = \theta^j(q_0)$ where $j \in \bN$ and where $q_0$ is an
element of $Q_\alpha$ which does not belong to the image of
$\theta$.  The latter condition on $q_0$ means that it is of the
form $q_0 = (0\,;\, a, b, 0)$ with $a,b\ge 1$ or $q_0 = (i\,;\, a,
0, c)$ with $a\ge 1$.  According to Proposition
\ref{prop:partitionE}, there exists exactly one representation of
$\alpha$ of that form and, by Proposition \ref{prop:degree_alpha},
it has degree $d(\alpha)$ and bi-degree $\ud(\alpha)$.  Thus $q_0$
is the element of $Q_\alpha$ of smallest size, and we can organize
$Q_\alpha$ in a sequence $\big(\theta^j(q_0)\big)_{j\ge 0}$ where
the size increases by steps of $1$  Along this sequence, the degree,
bi-degree and second partial degree are strictly increasing, since
for any $q=(i\,;\,a,b,c) \in Q_\alpha$, the formula
\eqref{def:theta} implies that $\ud(\theta(q)) = \ud(q) + (2f(i-1),
2f(i))$.
\end{proof}

The proof of the above proposition provides an explicit recursive
way of constructing the elements of $Q_\alpha$ by order of
increasing size: given any $q\in Q_\alpha$ the next element is
$\theta(q)$.  We will not use this explicit formula in the sequel,
except in the proof of the second corollary below.

\begin{corollary}
 \label{cor:char_size}
Let $d \in \bN^*$, $\ud\in\bN^2\setminus\{(0,0)\}$, and $\alpha\in
E^*$. If $\alpha\in E_d$, then $s_d(\alpha)$ is the size of the quad
of largest degree $\le d$ which represents $\alpha$.  If $\alpha\in
E_\ud$, then $s_\ud(\alpha)$ is the size of the quad of largest
bi-degree $\le \ud$ which represents $\alpha$.
\end{corollary}

\begin{proof}
Suppose that $\alpha\in E_d$.  Then, by Proposition \ref{prop:quad},
the size $s_d(\alpha)$ of $\alpha$ relative to $d$ is the largest
size achieved by a quad of degree $\le d$ in $Q_\alpha$. By
Proposition \ref{prop:Qalpha}, this is also the size of the quad of
largest degree $\le d$ in $Q_\alpha$.  The proof of the assertion in
bi-degree $\ud$ is similar.
\end{proof}

\begin{corollary}
 \label{cor:comp_size}
Let $d \in \bN^*$ and $\alpha\in E_d$.  There exists one and only
one representative $(i\,;\, a, b, c)$ of $\alpha$ which satisfies
\begin{equation}
 \label{ineq:comp_size}
 d-2f(i+1) < af(i) + bf(i+1) + cf(i+2) \le d.
\end{equation}
For this choice of quad, one has $s_d(\alpha)= a+b+c$.
\end{corollary}

\begin{proof}
By Corollary \ref{cor:char_size} and the remark following
Proposition \ref{prop:Qalpha}, the integer $s_d(\alpha)$ is the size
of the unique quad $q$ in $Q_\alpha$ satisfying $d(q) \le d <
d(\theta(q))$.  Upon writing $q = (i\,;\, a, b, c)$ and using the
formula \eqref{def:theta} for $\theta(q)$, the latter inequality
translates into \eqref{ineq:comp_size}, and the conclusion follows.
\end{proof}

\begin{proposition}
 \label{prop:Qd}
Let $\ud = (d_1,d_2) \in \bN^2\setminus\{(0,0)\}$.  The set $Q_\ud$
of all quads of bi-degree $\ud$ is a finite non-empty set whose
elements have distinct size.  If we order its elements by decreasing
size, then their sizes form a decreasing sequence of consecutive
integers while the points of $E_\ud$ that they represent form a
strictly decreasing sequence of positive real numbers. The element
of $Q_\ud$ of largest size has size $d_1+d_2$ and represents the
point $d_1+d_2/\gamma$, while the element of $Q_\ud$ of smallest
size represents the point $|d_1-d_2/\gamma|$, both points being of
bi-degree $\ud$.  The elements of $Q_\ud$ of intermediate sizes
represent points of bi-degree $\le (d_1,d_2-1)$.
\end{proposition}

\begin{proof}
The set $Q_\ud$ is not empty as it contains the quad
\[
 q_0
   = \begin{cases}
      (0\,;\, d_1,d_2,0) &\text{if $d_1>0$,}\\
      (1\,;\, d_2,0,0) &\text{if $d_1=0$.}
     \end{cases}
\]
Since $f(-2)=1$, $f(-1)=0$ and $f(j)\ge 1$ for each $j\ge 0$, the
formula \eqref{formulas:degree&bidegree} for the bi-degree shows
that $q_0$ is the only element of $Q_\ud$ if $d_1=0$ or if $d_2=0$.
As the proposition is easily verified in that case, we may assume
that $d_1$ and $d_2$ are positive.

Define $Q_\ud^+$ to be the set of quads $q=(i\,;\,a,b,c)$ of $Q_\ud$
with $b\ge 1$.  Then, under our present assumptions, $Q_\ud^+$ is
not empty as it contains the point $q_0$.  Moreover, the recurrence
relation for the function $f$ combined with
\eqref{formulas:degree&bidegree} shows that one defines a map
$\psi\colon Q_\ud^+ \to Q_\ud$ by sending a quad $q = (i\,;\,a,b,c)
\in Q_\ud^+$ to
\begin{equation}
 \label{formula:psi}
 \psi(q)
  = \begin{cases}
     (i\,;\,a-1,b-1,c+1) &\text{if $a\ge 2$,}\\
     (i+1\,;\,b-1,c+1,0) &\text{if $a = 1$ and $b\ge 2$,}\\
     (i+2\,;\,c+1,0,0) &\text{if $a = b = 1$.}
    \end{cases}
\end{equation}
This map is injective and decreases the size of a quad by $1$. Thus
any $q\in Q_\ud$ can be written in a unique way in the form $q =
\psi^j(q_0')$ where $j \in \bN$ and where $q_0'$ is an element of
$Q_\ud$ which does not belong to the image of $\psi$.  This means
that $q_0'$ is of the form $(0\,;\, a, b, 0)$ with $a,b\ge 1$ or
$(1\,;\, a, 0, 0)$ with $a\ge 1$, and thus that $q_0' = q_0$ since
its bi-degree is $\ud$.

The above discussion shows that we can organize $Q_\ud$ in a
sequence $\big(\psi^j(q_0)\big)_{j=0}^t$ where the size decreases by
steps of $1$, starting from the element $q_0$ of $Q_\ud$ of largest
size $d_1+d_2$, and ending with the element $q_t:=\psi^t(q_0)$ of
smallest size. Since the quad $q_t$ does not belong to $Q_\ud^+$, it
has the form $(i\,;\, a, 0, c)$ for some $i\ge 0$.  In the notation
of Proposition \ref{prop:partitionE}, it thus represents a point of
$E^{(i)} \subset E \setminus E^{(+)}$ which, by Proposition
\ref{prop:degree_alpha}, has bi-degree exactly $\ud$. Since $d_1 +
d_2/\gamma \in E^{(+)}$ and $|d_1 - d_2/\gamma| \in E \setminus
E^{(+)}$ are the only points of $E$ of bi-degree $\ud$, we conclude
that $q_0$ and $q_t$ are respectively the representations of $d_1 +
d_2/\gamma$ and $|d_1 - d_2/\gamma|$ coming from Proposition
\ref{prop:partitionE}. All intermediate quads $\psi^j(q_0)$ with
$j=1,\dots,t-1$ belong to $Q_\ud^+ \setminus \{q_0\}$.  They have
the form $(i\,;\,a,b,c)$ with $i=0$ and $a,b,c\ge 1$, or with $i\ge
1$ and $a,b\ge 1$. Therefore, by Proposition
\ref{prop:degree_alpha}, they represent points of bi-degree $\le
(d_1,d_2-1)$. Finally, for any given $q = (i\,;\, a, b, c) \in
Q_\ud^+$, the quad $\psi(q)$ given by \eqref{formula:psi} represents
the point $(a-1)\gamma^{-i} + (b-1)\gamma^{-i-1} +
(c+1)\gamma^{-i-2}$ of $E^*$ which, as a real number, is smaller
than the point $a\gamma^{-i} + b\gamma^{-i-1} + c\gamma^{-i-2}$
represented by $q$. Thus the points of $E^*$ represented by quads in
$Q_\ud$ decrease (in absolute value) with the size of these quads.
\end{proof}


\subsection{Proof of Theorem \ref{thm:comb_bideg}}
 \label{subsec:4.5}
We prove it in the following form.

\begin{theorem}
 \label{thm:card_Euds}
Let $\ud=(d_1,d_2) \in \bN^2$.  For each integer $s\ge 0$, define
\[
 E_\ud(s) = \{ \alpha\in E_\ud \,;\, s_\ud(\alpha) = s \}.
\]
Then, for $s>d_1+d_2$, this set is empty while for $0\le s\le
d_1+d_2$ its cardinality is
\begin{equation}
 \label{formula:card_Euds}
 \big| E_\ud(s) \big|
  = 2\min\{d_1,\, d_2,\, s,\, d_1+d_2-s\} + 1.
\end{equation}
\end{theorem}

\begin{proof}
We fix a choice of $d_1\ge 0$ and prove the theorem by recurrence on
$d_2\ge 0$.  For $d_2=0$, we have $E_{d_1,0}=\{0,1,\dots,d_1\}$ and
$s_{d_1,0}(i)=i$ for $i=0,1,\dots,d_1$.  Therefore $E_{d_1,0}(s)$
has cardinality $1$ for $0\le s \le d_1$ and is empty for $s>d_1$,
as asserted by the theorem.  Suppose now that $d_2>0$ and that the
statement of the theorem holds in bi-degree $(d_1,d_2-1)$.

Fix an integer $s\ge 0$.  In order to establish the formula in
bi-degree $(d_1,d_2)$, we first compare the sets $E_{d_1,d_2}(s)$
and $E_{d_1,d_2-1}(s)$.

1) According to Corollary \ref{cor:char_size}, the points of
$E_{d_1,d_2}(s)$ which do not belong to $E_{d_1,d_2-1}(s)$ are the
elements $\alpha$ of $E$ for which the quad of $Q_\alpha$ of largest
bi-degree $\le (d_1,d_2)$ has size $s$ but does not have bi-degree
$\le (d_1,d_2-1)$.  They are therefore the points of $E$ which are
represented by an element of $Q_{i,d_2}$ of size $s$ for some
integer $i$ with $0\le i\le d_1$.  Since by Proposition
\ref{prop:Qalpha} the quads representing the same point have
distinct second partial degrees, and since by Proposition
\ref{prop:Qd} each $Q_{i,d_2}$ contains at most one element of size
$s$, we conclude that the cardinality of $E_{d_1,d_2}(s) \setminus
E_{d_1,d_2-1}(s)$ is the number of indices $i$ with $0\le i\le d_1$
such that $Q_{i,d_2}$ contains an element of size $s$.

2) According again to Corollary \ref{cor:char_size}, the points of
$E_{d_1,d_2-1}(s)$ which do not belong to $E_{d_1,d_2}(s)$ are the
points $\alpha$ of $E_{d_1,d_2-1}$ for which $Q_\alpha$ contains
both a quad of size $s$ and bi-degree $\le (d_1,d_2-1)$, and a quad
of size $s+1$ and bi-degree $(i,d_2)$ for some $i$ with $0\le i\le
d_1$.  The first condition however is redundant because if $\alpha
\in E_{d_1,d_2-1}$ is represented by a quad of size $s+1$ and
bi-degree $(i,d_2)$ with $0\le i\le d_1$, then as the second partial
degree increases with the size in $Q_\alpha$ while the first partial
degree does not decrease (by Proposition \ref{prop:Qalpha}), the
quad of $Q_\alpha$ with largest bi-degree $\le (d_1,d_2-1)$ must
have size $s$.  Thus, the set $E_{d_1,d_2-1}(s) \setminus
E_{d_1,d_2}(s)$ consists of the points of $E_{d_1,d_2-1}$ which are
represented by an element of $Q_{i,d_2}$ of size $s+1$ for some $i$
with $0\le i\le d_1$. Moreover, according to Proposition
\ref{prop:Qd}, for a given $i\in\bN$ all elements of $Q_{i,d_2}$
represent points of bi-degree $\le (i,d_2-1)$ except for the ones of
smallest or largest size.  Therefore, the cardinality of
$E_{d_1,d_2-1}(s) \setminus E_{d_1,d_2}(s)$ is the number of indices
$i$ with $0\le i\le d_1$ such that $Q_{i,d_2}$ contains an element
of size $s$ and an element of size $s+2$.

Combining the conclusions of 1) and 2), we obtain that the
cardinality of $E_{d_1,d_2}(s)$ is equal to that of
$E_{d_1,d_2-1}(s)$ plus the number of indices $i$ with $0\le i\le
d_1$ such that $Q_{i,d_2}$ contains an element of size $s$ but no
element of size $s+2$. Since, by Proposition \ref{prop:Qd}, the
largest size of an element of $Q_{i,d_2}$ is $i+d_2$, the latter
condition on $i$ amounts to either $i+d_2=s$ or both $i+d_2=s+1$ and
$i\neq 0$ (so that $Q_{i,d_2}$ contains at least two elements and
thus contains an element of size $s$). This provides the recurrence
relation
\[
 |E_{d_1,d_2}(s)|
  = |E_{d_1,d_2-1}(s)|
    +
    \begin{cases}
     0 &\text{if\, $s<d_2$ or $s>d_1+d_2$,}\\
     1 &\text{if\, $s=d_1+d_2$,}\\
     2 &\text{if\, $d_2\le s < d_1+d_2$.}
    \end{cases}
\]
Combining this with the induction hypothesis for
$|E_{d_1,d_2-1}(s)|$, we get $|E_{d_1,d_2}(s)|=0$ if $s>d_1+d_2$ and
$|E_{d_1,d_2}(s)|=1$ if $s=d_1+d_2$.  If $s< d_1+d_2$, it also
provides the required value \eqref{formula:card_Euds} for
$|E_{d_1,d_2}(s)|$ because the difference
\[
 \min\{d_1,\, d_2,\, s,\, d_1+d_2-s\}
  - \min\{d_1,\, d_2-1,\, s,\, d_1+d_2-1-s\}
\]
is $0$ if $\min\{d_1,\, s\} < \min\{d_2,\, d_1+d_2-s\}$ and $1$
otherwise. Since $\min\{d_2,\, d_1+d_2-s\} = d_2 - s + \min\{d_1,\,
s\}$, this difference is therefore $0$ if $s<d_2$ and is $1$ if
$d_2\le s< d_1+d_2$.
\end{proof}


\subsection{Proof of Theorem \ref{thm:comb_deg}}
 \label{subsec:4.6}
Similarly, we prove Theorem \ref{thm:comb_deg} in the following
form.

\begin{theorem}
 \label{thm:card_Eds}
Let $d \in \bN$.  For each integer $s\ge 0$, define
\[
 E_d(s) = \{ \alpha\in E_d \,;\, s_d(\alpha) = s \}.
\]
Then, for $s>d$, this set is empty while for $0\le s\le d$ its
cardinality is
\[
 \big| E_d(s) \big|
  = \begin{cases}
     2s+1 &\text{if $0\le s<d$,}\\
     d+1  &\text{if $s=d$.}
    \end{cases}
\]
\end{theorem}

\begin{proof}
We proceed by recurrence on $d$.  For $d=0$, we have $E_0=\{0\}$ and
since $s_0(0)=0$, the theorem is verified in that case. Suppose now
that $d > 0$ and that the conclusion of the theorem holds in smaller
degree.

Fix an integer $s\ge 0$.  Arguing in a similar way as in the proof
of Theorem \ref{thm:card_Euds}, we find that:
\begin{itemize}
\item[1)] $E_d(s) \setminus E_{d-1}(s)$ consists of the
points of $E$ which are represented by a quad of $Q_{i,d-i}$ of size
$s$ for some integer $i$ with $0\le i\le d$; its cardinality is the
number of indices $i$ with $0\le i\le d$ such that $Q_{i,d-i}$
contains an element of size $s$;
\item[2)] $E_{d-1}(s) \setminus E_d(s)$ consists of the points of $E_{d-1}$
which are represented by an element of $Q_{i,d-i}$ of size $s+1$ for
some $i$ with $0\le i\le d$; its cardinality is the number of
indices $i$ with $0\le i\le d$ such that $Q_{i,d-i}$ contains an
element of size $s$ and an element of size $s+2$.
\end{itemize}
Thus, the cardinality of $E_d(s)$ is equal to that of $E_{d-1}(s)$
plus the number of indices $i$ with $0\le i\le d$ such that
$Q_{i,d-i}$ contains an element of size $s$ but no element of size
$s+2$. Since the largest size of an element of $Q_{i,d-i}$ is $d$,
the latter condition amounts to either $d=s$ or both $d=s+1$ and
$i\neq 0$. This gives the recurrence relation
\[
 |E_d(s)|
  = |E_{d-1}(s)|
    +
    \begin{cases}
     0 &\text{if\, $s<d-1$ or $s>d$,}\\
     d &\text{if\, $s=d-1$,}\\
     d+1 &\text{if\, $s=d$,}
    \end{cases}
\]
and from there the conclusion follows.
\end{proof}

%
%

\section{Application to a dimension estimate}
 \label{sec:applic}

The following result illustrates how the theory developed in
Sections \ref{sec:ring} and \ref{sec:comb} can be used to derive
dimension estimates of the type that one requires in the
construction of auxiliary polynomials.

\begin{theorem}
Let $d\in\bN^*$ and\/ $\delta\in\bR$ with $0<\delta\le \gamma d$.
Define $V_d(\delta)$ to be the set of sequences $\gA \in \QgX_{\le
d}$ satisfying $|\gA| \ll |\gX_0^{(0)}|^{\delta}$. Then,
$V_d(\delta)$ is a subspace of $\bQ[\gX^{(0)}, \gX^{(-1)}]_{\le d}$
and its dimension satisfies
\begin{equation}
 \label{applic:ineq_thm}
 c_1(d\delta)^{3/2}
 \le
 \dim_\bQ V_d(\delta)
 \le
 1+c_2(d\delta)^{3/2},
\end{equation}
for appropriate positive constants $c_1$ and $c_2$ depending only on
$\xi$.
\end{theorem}

\begin{proof}
It is clear that $V_d(\delta)$ is a vector space over $\bQ$.
Combining Lemma \ref{lemma:comb_gM} with Theorem \ref{thm:basis_gM}
and then using Corollary \ref{cor:comp_size} we find that his
dimension is
\[
 \dim_\bQ V_d(\delta)
 = \sum_{\{\alpha\in E_d\,;\, \alpha\le\delta\}}
   (2s_d(\alpha)+1)
 = 1 + \sum_S (2(a+b+c)+1),
\]
where the rightmost sum runs over the set $S$ of all quads
$(i\,;\,a,b,c)$ satisfying the system of inequalities
\begin{gather}
 \label{applic:ineq1}
 d-2f(i+1) < af(i) + bf(i+1) + cf(i+2) \le d, \\
 \label{applic:ineq2}
 a\gamma^{-i} + b\gamma^{-i-1} + c\gamma^{-i-2} \le \delta.
\end{gather}
For each $i\in\bN$, let $S_i$ denote the set of triples
$(a,b,c)\in\bN^3$ satisfying both $a\ge 1$ and the first condition
\eqref{applic:ineq1}.  In the computations below, we freely use the
fact that $\gamma^{i-1}\le f(i) \le \gamma^i$ for each $i\in\bN$ (as
one easily shows by recurrence on $i$).

Let $i\in\bN$.  We first note that $S_i$ is empty if $f(i)>d$ (since
we require $a\ge 1$).  Assume that $f(i)\le d$.  Then each
$(a,b,c)\in S_i$ satisfies $d/(2f(i+2)) \le a+b+c\le d/f(i)$ and so
\begin{equation*}
 \begin{aligned}
 \frac{d}{\gamma^{i+2}} \le \frac{d}{f(i+2)}
 &\le 2(a+b+c)+1
 \le \frac{2d}{f(i)}+1 \le \frac{3d}{f(i)}
 \le \frac{3d}{\gamma^{i-1}},\\
 \frac{d}{2\gamma^{2i+4}} \le \frac{d}{2f(i+2)}\gamma^{-i-2}
 &\le a\gamma^{-i} + b\gamma^{-i-1} + c\gamma^{-i-2}
 \le \frac{d}{f(i)}\gamma^{-i}
 \le \frac{d}{\gamma^{2i-1}}.
 \end{aligned}
\end{equation*}
The second chain of inequalities implies that if some $(a,b,c)\in
S_i$ satisfies \eqref{applic:ineq2}, then we must have $d\le
2\gamma^{2i+4}\delta$.  On the other hand, it also implies that any
$(a,b,c)\in S_i$ satisfies \eqref{applic:ineq2} as soon as $d\le
\gamma^{2i-1}\delta$.  Therefore, we obtain
\[
 1 + \sum_{i\in J} |S_i| \frac{d}{\gamma^{i+2}}
 \le
 \dim_\bQ V_d(\delta)
 \le
 1 + \sum_{i\in I} |S_i| \frac{3d}{\gamma^{i-1}}
\]
where $I$ denotes the set of integers $i\ge 0$ such that $f(i)\le d$
and $d \le 2\gamma^{2i+4}\delta$, and $J$ the set of integers $i\ge
0$ such that $f(i)\le d$ and $d\le\gamma^{2i-1}\delta$.

Again, let $i\in\bN$.  For each choice of integers $a\ge 1$ and
$c\ge 0$ with $af(i)+cf(i+2)\le d$, there are exactly two choices of
integer $b\ge 0$ such that $(a,b,c)$ satisfies \eqref{applic:ineq1},
or equivalently such that $(a,b,c)\in S_i$.  This means that the
cardinality $|S_i|$ of $S_i$ is twice the number of points
$(a,c)\in\bN^*\times \bN$ satisfying $af(i)+cf(i+2)\le d$.  Since
$f(i)\le d$ for each $i$ in $I$ or $J$, this number is non-zero and
a short computation provides absolute constants $c_3>0$ and $c_4>0$
such that $|S_i| \le c_3 d^2 \gamma^{-2i}$ for each $i\in I$ and
$|S_i| \ge c_4 d^2 \gamma^{-2i}$ for each $i\in J$.

If $I$ is not empty, it contains a smallest element $i_0$ and so the
above considerations give
\[
 \dim_\bQ V_d(\delta)
 \le 1 + 3 c_3 d^3 \sum_{i=i_0}^\infty \gamma^{-3i+1}
 \le 1 + 4 \gamma c_3 (d\gamma^{-i_0})^3.
\]
Since $d\le 2 \gamma^{2i_0+4}\delta$, we find that $d\gamma^{-i_0}
\le (2\gamma^4 d\delta)^{1/2}$ and so $\dim_\bQ V_d(\delta) \le 1 +
c_2 (d\delta)^{3/2}$ with $c_2 = 12\gamma^7 c_3$.  If $I$ is empty,
$V_d(\delta)$ has dimension $1$ and this inequality still holds.

Let $j_0$ denote the integer for which $\gamma^{2j_0-3}\delta < d
\le \gamma^{2j_0-1}\delta$.  Since $\delta \le \gamma d$, we have
$j_0\ge 0$, and thus $j_0$ belongs to $J$ if and only if $f(j_0) \le
d$. In that case, using $\gamma^{2j_0-3}\delta < d$, we find that
$d\gamma^{-j_0} \ge \gamma^{-3/2} (d\delta)^{1/2}$ and so
\[
 \dim_\bQ V_d(\delta)
 \ge |Q_{j_0}| \frac{d}{\gamma^{j_0+2}}
 \ge \gamma^{-2} c_4 (d\gamma^{-j_0})^3
 \ge \gamma^{-7} c_4 (d\delta)^{3/2}.
\]
If $j_0\notin J$, then we have $d < f(j_0) \le \gamma^{j_0}$, thus
$d\delta \le \gamma^{j_0}\delta \le \gamma^{-j_0+3} d \le \gamma^3$
and therefore
\[
 \dim_\bQ V_d(\delta)
 \ge
 1
 \ge
 \gamma^{-9/2} (d\delta)^{3/2},
\]
showing that the lower bound in \eqref{applic:ineq_thm} holds with
$c_1 = \min\{ \gamma^{-7} c_4,\, \gamma^{-9/2}\}$.
\end{proof}

%
%

\section{A complementary result}
 \label{sec:complement}

Following a suggestion of Daniel Daigle, we show:

\begin{theorem}
 \label{thm:prime}
The ideal $I$ defined in Lemma \ref{lemma:I} is a prime ideal of
rank $3$ of the ring $R=\QX$.
\end{theorem}

This provides a proof of Corollary \ref{cor:ker_pi} which is
independent of the combinatorial arguments of Section
\ref{sec:comb}.  Indeed, it follows easily from the considerations
of \S\S\ref{subsec:2.4}--\ref{subsec:2.5} that $\gX_0^{(0)}$,
$\gX_0^{(-1)}$ and $\gX_1^{(-1)}$ are elements of $\QgX$ which are
algebraically independent over $\bQ$.  On the other hand, the
evaluation map $\pi\colon R \to \QgX$ defined by \eqref{def:pi}
induces a surjective ring homomorphism $\bar{\pi}\colon R/I \to
\QgX$ and, if we take for granted Theorem \ref{thm:prime}, the
quotient $R/I$ is an integral domain of transcendence degree $3$
over $\bQ$. Therefore, as Daniel Daigle remarked, this means that
$\bar{\pi}$ is an isomorphism.  This not only proves Corollary
\ref{cor:ker_pi} but also:

\begin{corollary}
The ring $\QgX$ is an integral domain of transcendence degree $3$
over $\bQ$.
\end{corollary}

In order to prove Theorem \ref{thm:prime}, we first note that we may
assume, without loss of generality, that the coefficient $a_{2,2}$
of the matrix $M$ is non-zero.  Indeed, as we saw at the end of
\S\ref{subsec:2.2}, at least one of the coefficients $a_{1,1}$ or
$a_{2,2}$ of $M$ is non-zero.  If $a_{2,2}=0$, then $a_{1,1}\neq 0$
and we replace $I$ by its image under the ring automorphism of $R$
which sends $X_i$ to $X_{2-i}$ and $X^*_i$ to $X^*_{2-i}$ for
$i=0,1,2$.  This automorphism fixes the first two generators
$\det(\uX)-1$ and $\det(\uX^*)-1$ of $I$ and maps $\Phi(\uX,\uX^*)$
to a polynomial of the same form with the coefficient $a_{2,2}$
replaced by $-a_{1,1}\neq 0$.

Now, let $V$, $V^*$ and $W$ be indeterminates over $R$.  We put a
$\bN^3$-grading on the ring $R_3 := R[V,V^*,W]$ by requesting that
each variable is multi-homogeneous with multi-degree:
\[
 \begin{gathered}
 \deg(X_i)=(1,0,i),\quad \deg(X_i^*)=(0,1,i) \quad \text{for $i=0,1,2$,}\\
 \deg(V)=(1,0,1),\quad \deg(V^*)=(0,1,1) \et \deg(W)=(0,0,1).
 \end{gathered}
\]
A polynomial $P$ in $R_3$ is thus multi-homogeneous of multi-degree
$(d_1,d_2,d_3)$ if and only if, in the usual sense, it is
homogeneous of degree $d_1$ in $(\uX,V)$, homogeneous of degree
$d_2$ in $(\uX^*,V^*)$, and if its image under the specialization
$X_i\mapsto W^iX_i$, $X_i^*\mapsto W^iX^*_i$ ($i=0,1,2$), $V\mapsto
WV$ and $V^*\mapsto WV^*$ belongs to $W^{d_3}R[V,V^*]$. In this
case, $d_3$ is called the \emph{weight} of~$P$.

Let $I_3$ denote the ideal of $R_3$ generated by
\[
 \begin{aligned}
 F &= \det(\uX) - V^2,\\
 F^* &= \det(\uX^*)- (V^*)^2, \\
 G &=
    a_{1,1} \left| \begin{matrix}
            X^*_0 &X^*_1 \\ X_0 &X_1
            \end{matrix}
     \right| W^2
 + \left(
   a_{1,2} \left| \begin{matrix}
            X^*_1 &X^*_2 \\ X_0 &X_1
            \end{matrix}
     \right|
 + a_{2,1} \left| \begin{matrix}
            X^*_0 &X^*_1 \\ X_1 &X_2
            \end{matrix}
     \right|
   \right) W
 + a_{2,2} \left| \begin{matrix}
            X^*_1 &X^*_2 \\ X_1 &X_2
            \end{matrix}
     \right|.
 \end{aligned}
\]
Since $F$, $F^*$ and $G$ are respectively multi-homogeneous of
multi-degree $(2,0,2)$, $(0,2,2)$ and $(1,1,3)$, the ideal $I_3$ is
multi-homogeneous. By construction, it is mapped to $I$ under the
$R$-linear ring homomorphism from $R_3$ to $R$ sending $V$, $V^*$
and $W$ to $1$. Moreover, any element of $I$ is the image of a
multi-homogeneous element of $I_3$ under that map.  Therefore, in
order to prove Theorem \ref{thm:prime}, it suffices to show that
$I_3$ is a prime ideal of $R_3$.  In preparation to this, we first
establish the following lemma where, for any $f\in R_3$, we define
$(I_3\colon f) = \{a\in R_3\,;\, af\in I_3 \}$.

\begin{lemma}
 \label{lemma:reg_seqVW}
The polynomials $F$, $F^*$ and $G$ form a regular sequence in $R_3$
and we have $(I_3\colon X_0) = (I_3\colon X^*_0) = I_3$.
\end{lemma}

\begin{proof}
Each polynomial in the sequence $X_0$, $X^*_0$, $F$, $F^*$, $G$
depends on a variable on which the preceding polynomials do not
depend ($V$ for $F$, $V^*$ for $F^*$ and $W$ for $G$).  Therefore,
this sequence generates an ideal of $R_3$ of rank $5$, and so it is
a regular sequence.  Since these polynomials are homogeneous, any
reordering of this sequence remains a regular sequence, and the
conclusion follows.
\end{proof}

We now complete the proof of Theorem \ref{thm:prime} by showing:

\begin{lemma}
The ideal $I_3$ is prime of rank $3$.
\end{lemma}

\begin{proof}
Let $S$ denote the multiplicative subset of $R_3$ generated by $X_0$
and $X_0^*$.  By Lemma \ref{lemma:reg_seqVW}, we have $(I_3\colon f)
= I_3$ for each $f\in S$.  So, it is equivalent to prove that
$S^{-1}I_3$ is a prime ideal in the localized ring $S^{-1} R_3$.
Since
\[
 X_0^{-1}F =  X_2-X_0^{-1}(X_1^2+V^2)
 \et
 (X_0^*)^{-1}F^* = X_2^*-(X_0^*)^{-1}((X_1^*)^2+(V^*)^2),
\]
this amounts simply to showing that $G$ is mapped to a prime element
under the ring endomorphism of $S^{-1} R_3$ sending $X_2$ to
$X_0^{-1}(X_1^2+V^2)$, $X_2^*$ to $(X_0^*)^{-1}
((X_1^*)^2+(V^*)^2)$, and all other variables to themselves.  The
image of $G$ takes the form $(X_0X^*_0)^{-1} H$ where
\[
 \begin{aligned}
 H =
    a_{1,1}
    &\left| \begin{matrix}
            X^*_0 &X^*_1 \\ X_0 &X_1
            \end{matrix}
     \right| X_0 X^*_0 W^2
  +  a_{1,2} \left| \begin{matrix}
            X^*_0X^*_1 &(X^*_1)^2+(V^*)^2 \\ X_0 &X_1
            \end{matrix}
     \right| X_0 W \\
 &+ a_{2,1} \left| \begin{matrix}
            X^*_0 &X^*_1 \\ X_0 X_1 &X_1^2+V^2
            \end{matrix}
     \right| X^*_0 W
 + a_{2,2} \left| \begin{matrix}
            X^*_0X^*_1 &(X^*_1)^2+(V^*)^2 \\ X_0 X_1 &X_1^2+V^2
            \end{matrix}
     \right|.
 \end{aligned}
\]
Since $R_3$ is a unique factorization domain, we are reduced to
showing that $H$ is an irreducible element of $R_3$.  Moreover,
since $H$ is multi-homogeneous of multi-degree $(2,2,3)$, it
suffices to prove that $H$ has no non-constant multi-homogeneous
divisor of multi-degree $<(2,2,3)$.  Let $H_0$ denote the constant
coefficient of $H$, viewed as a polynomial in $W$.  Since
$a_{2,2}\neq 0$, it is non-zero.  It takes the form $H_0 =
a(V^*)^2-b$ where $a$ and $b$ are relatively prime elements of
$R[V]$ such that $ab$ is not a square in $R[V]$ ($ab$ is divisible
by $X_0$ but not by $X_0^2$). Therefore $H_0$ is irreducible. If $A$
is a multi-homogeneous divisor of $H$ of multi-degree $<(2,2,3)$,
then the constant coefficient $A_0$ of $A$ (as a polynomial in $W$)
is a divisor of $H_0$.  Since $H_0$ is irreducible and
multi-homogeneous of the same multi-degree $(2,2,3)$ as $H$, it
follows that $A_0$ is a constant and therefore that $A$ itself is a
constant.
\end{proof}

\begin{acknowledgments}
The authors thank Daniel Daigle for several interesting discussions
on the topic of this paper, for the suggestion mentioned above, as
well as for suggesting the formalism of the ring of sequences $\gS$
in Section 2.
\end{acknowledgments}


\end{document}